\documentclass{article}
\usepackage{latexsym}
\usepackage{amssymb}
\usepackage{amsmath}%
\newtheorem{theorem}{Theorem}
\newtheorem{lemma}[theorem]{Lemma}

\newtheorem{prop}{Proposition}
\newcommand\QED{\ifhmode\allowbreak\else\nobreak\fi\quad\nobreak$\Box$\medbreak}
\newenvironment{proof}{\par\noindent{\bf Proof:}}{\rm\enspace{\QED\par}}

\begin{document}
\title{Period Lengths for Iterated  Functions.\\
(Preliminary Version)}
\author{Eric Schmutz\\
{
 Eric.Jonathan.Schmutz@drexel.edu}\\
{ Drexel Math Department and University of}\\ { Delaware
Mathematical Sciences Department.}\\
 }
 \maketitle
 \abstract{ Let $\Omega_{n}$ be the $n^{n}$-element set
consisting of  functions that have $[n]$
 as both domain and codomain.
 Since  $\Omega_{n}$ is finite,  it
 is clear by the pigeonhole principle that, for any $f\in\Omega_{n}$,
 the sequence of compositional iterates
 \[ f,f^{(2)},f^{(3)},f^{(4)}\dots \]
 must eventually repeat. Let ${\bf T}(f)$ be the period of this
 eventually periodic sequence of functions,
  i.e. the least positive integer $T$ such that, for all $m\geq n$,
 \[ f^{(m+T)}= f^{(m)}.\]
 \par
 A closely related number ${\bf B}(f)=$ the product of the
lengths of the cycles of $f$, has previously been used as an
approximation for ${\bf T}.$  This paper proves that the average
values of these two quantities are quite different. The expected
value of ${\bf T}$ is
 \[
 \frac{1}{n^{n}}\sum\limits_{f\in \Omega_{n}}{\bf T}(f)=
\exp\left( k_{0}\sqrt[3]{ \frac{n}{ \log^{2}
n}}\bigl(1+o(1)\bigr)\right),
\] where $k_{0}$ is a complicated but explicitly defined
constant that is approximately 3.36.  The expected value of ${\bf
B}$ is much larger:
\[ \frac{1}{n^{n}}\sum\limits_{f\in \Omega_{n}}{\bf B}(f)=\exp\left(
\frac{3}{2}\sqrt[3]{n}(1+o(1))\right).
\]
 }
 \section{Introduction}
 Let $\Omega_{n}$ be the $n^{n}$-element set consisting
of  functions that have $[n]$
 as both domain and codomain, and let $f^{(t)}$ denote $f$
 composed with itself $t$ times.
 Since  $\Omega_{n}$ is finite,  it
 is clear (by the pigeonhole principle) that, for any $f\in\Omega_{n}$,
 the sequence of compositional iterates
 \[ f,f^{(2)},f^{(3)},f^{(4)}\dots \]
 must eventually repeat. Define  ${\bf T}(f)$ to be the period of
 this sequence, i.e. the least $T$ such that, for all $m\geq n$,
 \[ f^{(m+T)}= f^{(m)}.\]
\par We say $v\in [n]$ is a {\sl cyclic vertex} if there is a $t$ such that
$f^{(t)}(v)=v.$  The  restriction of $f$ to cyclic vertices is a
permutation of the cyclic vertices,  and the period ${\bf T}$ is
just the order of this permutation, i.e. the least common multiple
of the cycle lengths.\par
 \par Harris showed that ${\bf T}(f)=
e^{\frac{1}{8}\log^{2}n(1+o(1))}$ for most functions $f$.  To make
this precise,  let $P_{n}$ denote the uniform distribution on
$\Omega_{n}$; $P_{n}(\lbrace f\rbrace)=n^{-n}$ for all $f$. Define
$a_{n}=\frac{1}{8}\log^{2}n,$let $b_{n}=\frac{1}{\sqrt{24}}
\log^{3/2}n,$  and let $\phi(x)=
\frac{1}{\sqrt{2\pi}}\int\limits_{-\infty}^{x}e^{-t^{2}/2}dt.$
Using Erd\H os and Tur\' an's seminal results\cite{ET3}, Harris
proved
\begin{theorem}(Harris\cite{harris2})
\label{harris}For any fixed $x$, \[\lim\limits_{n\rightarrow
\infty} P_{n}\left( \frac{\log {\bf T}-a_{n}}{b_{n} } \leq
x\right)=\phi(x).
 \]
\end{theorem}
 {\bf Comment:} Harris actually stated his theorem for a closely
related random variable ${\bf O}(f)=$ the number of distinct
functions in the sequence $f,f^{(2)},f^{(3)},\dots $. However it
is clear from his proof that Theorem \ref{harris} holds too. In
fact, it is straightforward to verify that, for all $f\in
\Omega_{n}$, \ $|{\bf O}(f)-{\bf T}(f)|<n.$ Such inequalities have
been proved by D\' enes \cite{Denes}.
\par
 Let  ${\bf B}(f)$ be the product, with multiplicities, of the
lengths of the cycles of $f$. Obviously ${\bf T}(f)\leq {\bf
B}(f)$ for all $f$,  and for some
 exceptional functions  ${\bf B}(f)$
is {\sl much} larger than ${\bf T}(f)$. For example, if $f$ is a
permutation with $n/2$ cycles of length $2$, then ${\bf
B}(f)=2^{n/2}$, but ${\bf T}(f)=2.$ In fact, the maximum value
${\bf T}$ can have is $e^{\sqrt{n\log n}(1+o(1))}$ \cite{Nicolas}.
However,  for most random mappings $f\in \Omega_{n}$,  $ {\bf
B}(f)$ is a reasonably good approximation for $ {\bf T}(f).$  For
example, the following proposition follows easily from results of
Arratia and Tavare \cite{AT}.
\begin{prop}
\label{typical}
 Let  $\ell_{n},n=1,2,\dots $ be any sequence of positive numbers.
There is a constant $c>0$ such that, for all sufficiently large
$n$,
\[ P_{n}(\log {\bf B}-\log {\bf T}\geq \ell_{n} )\leq \frac{c\log n (\log\log
n)^{2}}{\ell_{n}}.\]
\end{prop}
Although $\log {\bf B}(f)$ and  $\log {\bf T}(f)$ are
approximately equal  for most functions
   $f$,the
set of exceptional functions is nevertheless large enough so that
the expected values of the two random variables ${\bf B}$ and
${\bf T}$ are quite different. The following theorem will be
proved.
\begin{theorem}
\label{muB} $  E_{n}({\bf
B})=\exp\left(\frac{3}{2}\sqrt[3]{n}+O(n^{1/6})\right).$
\end{theorem}
To state a corresponding theorem for ${\bf T}$, we need to define
a constant $k_{0}$.
First define
$I=\int\limits_{0}^{\infty}\log\log(\frac{e}{1-e^{-t}})dt,$ and
define $\beta_{0}=\sqrt{8I}$. This constant first appears first in
\cite{gs} where it is proved that the expected order of a random
permutation is $\exp\left(\beta_{0}\sqrt{ n/\log
n}(1+o(1))\right).$ Also define
 \begin{equation}
 \label{kdef}
 k_{0}=\frac{3}{2}(3I)^{2/3}.
 \end{equation}
We prove
\begin{theorem}
\label{muT} \[ E_{n}
 ({\bf T})=
\exp\left( k_{0}\sqrt[3]{ \frac{n}{ \log^{2}
n}}\bigl(1+o(1)\bigr)\right).\]
\end{theorem}

\par The paper is organized as follows.
Section 2 reviews some relevant graph theory. In particular, we
need  Renyi's  asymptotic approximations for the expected length
of the cycle in a random $f$ that is chosen uniformly from those
$f\in \Omega_{n}$ that have one cycle.  In section 3, we apply
Hansen's deconditioning argument so that estimation of $E_{n}({\bf
B})$ can be reduced to a simpler problem of estimating the
coefficients of a particular generating function $E_{z}({\bf B}).$
In section 4, we use a Tauberian theorem of Odlyzko to estimate
the coefficients of $E_{z}({\bf B})$ and complete the proof of
Theorem \ref{muB}. In section 5, estimates for $E_{n}({\bf T})$
are deduced from earlier work on the expected order of a random
permutation.

\section{ The cycle length for a connected map.}
 If $f\in\Omega_{n}$, let $D_{f}$ be the directed graph with
vertex set $[n]$ and a directed edge from $v$ to $w$ if and only
if $f(v)=w$.
 Provided
we count loops and cycles of length 2, the  weak components of
$D_{f}$ each have exactly one directed  cycle.  Let $U_{n}\subset
\Omega_{n}$ be the set of functions for which $D_{f}$ has exactly
one weak component.  For $f\in U_{n}$, let $L_{n}(f)$ be the
length of the unique cycle $f$ has.  Let $\kappa_{n}$ be the
expected value of $L_{n}$ for a uniform random $f\in U_{n}.$  We
will need asymptotic estimates for $\kappa_{n}$. Apparently Renyi
did this calculation in \cite{renyi}; see page 366
 of Bollob\' as \cite{Bol}.
\begin{lemma} (Renyi)
\label{kappan}
  \[ |U_{n}|=\sum\limits_{k}\binom{n}{k}(k-1)! kn^{n-1-k}=
n^{n}\sqrt{\frac{\pi}{2n}}(1+O(\frac{1}{\sqrt{n}})),\] and
%
  \[ \kappa_{n}=\frac{1}{|U_{n}|}\sum\limits_{k}k\binom{n}{k}(k-1)! kn^{n-1-k}
=\sqrt{\frac{2n}{\pi}}(1+O(\frac{1}{\sqrt{n}})).\]
\end{lemma}
Additional background material on random mappings can be found in
\cite{FO}.

 \section{Deconditioning}
The proof of Theorem \ref{muT} relies on a probabilistic technique
that Jennie Hansen developed for random mappings in \cite{jh1}.
Let $P_{n}$ be the uniform probability measure on $\Omega_{n}$,
and let $\Omega= \bigcup\limits_{n=1}^{\infty}\Omega_{n}$. For any
$f\in \Omega$ and any positive integer $d$, let $\alpha_{d}(f)=$
the number of $d$-vertex weak components that the graph $D_{f}$
has. Given $z\in (0,1)$, and $A\subseteq \Omega$, define
$P_{z}(A)$ to be the probability that the following procedure
selects an element of $A$:
\begin{itemize}
\item Choose independently the numbers
$\alpha_{1},\alpha_{d},\alpha_{3},\dots,$ where $\alpha_{d}$ has a
Poisson distribution with mean
$\lambda_{d}=\frac{z^{d}}{d}\left(e^{-d}\sum\limits_{k=0}^{d-1}\frac{d^{k}}{k!}\right).$
\item Let $\nu=\sum\limits_{d}d\alpha_{d.}$
\item Pick $f$ uniform randomly from among all $f\in
\Omega_{\nu}$ having $\alpha_{d}$ components of size
$d,d=1,2,3\dots.$
\end{itemize}
 This defines a probability measure $P_{z}$ on
$\Omega$. For any $\Psi:\Omega\rightarrow \Re,$ let
$E_{z}(\Psi)=\sum\limits_{f}\Psi(f)P_{z}(\lbrace f\rbrace)$ denote
its $P_{z}$ expected value, and let
$E_{n}(\Psi)=E_{z}(\Psi|\nu=n)=\frac{1}{n^{n}}\sum\limits_{f\in
\Omega_{n} }{\bf T}(f).$
 Hansen
proved the following theorem
\begin{theorem}(\cite{jh1})
\label{jhthm} If $\Psi$ is any function on $\Omega$ that is
determined by the sequence $\alpha_{1},\alpha_{2},\dots $,  then
\[
E_{n}(\Psi)=\frac{e^{n}n!}{n^{n}}[[z^{n}]]H(z)\frac{E_{z}(\Psi)}{1-z},\]
where $H(z)=\sum\limits_{m=0}^{\infty}\frac{m^{m}}{m!}(z/e)^{m}$
and \lq\lq\ $ [[z^{n}]] $\rq\rq  means \lq\lq the coefficient of
$z^{n} $ in $\dots $\rq\rq .
\end{theorem}
The point of this construction is that $E_{z}(\Psi)$ is much
easier to compute that $E_{n}(\Psi)$ because the variables
$\alpha_{d}, d=1,2,\dots$ are independent with respect to $P_{z}$
and not with respect to $P_{n}.$ It is quite analagous to Shepp
and  Lloyd's work on random permutations \cite{SheppL}.
\par
For the problem under consideration, we let
 $B_{d}=$ the  product of the lengths of cycles  in $d-$vertex
 components ( If $f$ has no $d$-vertex components, then
 $B_{d}=1$).
Thus $ {\bf B}=\prod\limits_{d=1}^{\infty}B_{d}$ is a product of
$P_{z}$-independent random variables, and
\begin{equation}
\label{product} E_{z}({\bf B})=\prod\limits_{d=1}^{\infty}E_{z}({
B}_{d}).
\end{equation}
Evaluating the $d$'th term in this product, we have
\begin{equation}
E_{z}(B_{d})  =\sum\limits_{m}P_{z}(\alpha_{d}=m)E_{z}({
 B}_{d}|\alpha_{d}=m).
 \end{equation}
 Note that
 \begin{equation}
 \label{dterm}
 E_{z}({
 B}_{d}|\alpha_{d}=m)= \kappa_{d}^{m}\end{equation}
 where
 $\kappa_{d}=$ average length of the cycle  for a uniform random connected map
 on
 $d$ vertices.
Therefore
\begin{equation}
\label{bdterm}
E_{z}(B_{d})=e^{-\lambda_{d}+\lambda_{d}\kappa_{d}}.
\end{equation}
Combining (\ref{bdterm}), (\ref{product}), and Theorem
\ref{jhthm}, we get
 \begin{equation}
 E_{n}({\bf
B})=\frac{n!e^{n}}{n^{n}}
[[z^{n}]]H(z)\frac{1}{1-z}\exp\left(\sum\limits_{d=1}^{\infty}\lambda_{d}(\kappa_{d}-1)\right).
 \end{equation}
 To simplify notation later, define $c_{d}$ so that
 $c_{d}z^{d}=\lambda_{d}(\kappa_{d}-1),$ i.e.
 \begin{equation}
 \label{cds}
 c_{d}=\frac{(\kappa_{d}-1)}{d}(e^{-d}\sum\limits_{k=0}^{d-1}\frac{d^{k}}{k!})=\frac{1}{\sqrt{2\pi d}}+
 O(\frac{1}{d}).
 \end{equation}
( Further information of the approximation
$e^{-d}\sum\limits_{k=0}^{d-1}\frac{d^{k}}{k!}=\frac{1}{2}+O\frac{1}{\sqrt{d}})$
can be found in \cite{Flaj}.)
 Let
$\mu(m)=[[z^{m}]]\frac{1}{1-z}\exp\left(\sum\limits_{d=1}^{\infty}c_{d}z^{d})\right)
,$ and recall that, for all $m$,
$[[z^{m}]]H(z)=\frac{m^{m}}{m!e^{m}}.$ Then
\begin{equation}
\label{AB}
 E_{n}({\bf B})=\frac{n!e^{n}}{n^{n}}
\sum\limits_{m=0}^{n}\mu(m)\frac{(n-m)^{n-m}}{(n-m)! e^{n-m}}.
\end{equation}
In the next section we derive an asymptotic formula for $\log
\mu(n)$, and then use it to estimate the right side of (\ref{AB}).

\section{Tauberian theorem}\

In this section, we derive an asymptotic formula for $\log
\mu(n)$. Let $a_{n}=\mu(n)-\mu(n-1), $ and  for $s>0$, let
$F(s)=\sum\limits_{n}a_{n}e^{-ns}=
\exp\left(\sum\limits_{d=1}^{\infty}c_{d}e^{-ds}\right).$ Also let
$g(s)=\log F(s)=\sum\limits_{d=1}^{\infty}c_{d}e^{-ds}.$ To apply
Odlyzko's Tauberian theorem\cite{amo},  we will need asymptotic
estimates for the $j$'th derivative
$g^{(j)}(s)=\sum\limits_{d=1}^{\infty}(-d)^{j}c_{d}e^{-ds}$ for
$j=0,1,2,3$. From (\ref{cds}), we have
\begin{equation}
 g^{(j)}(s)=\sum\limits_{d=1}^{\infty}
 \left(\frac{(-d)^{j}e^{-ds}}{\sqrt{2\pi
 d}}+O(\frac{d^{j}e^{-ds}}{d}) \right).
\end{equation}
Hence, as $s\rightarrow 0^{+},$
\[
 g^{(j)}(s)=(-1)^{j}\int\limits_{0}^{\infty}
 \frac{t^{j}e^{-ts}}{\sqrt{2\pi t}}  dt
  +O\left( \int\limits_{0}^{\infty}t^{j-1}e^{-ts}dt\right)
\]

\begin{equation}
\label{gs}
=\frac{(-1)^{j}\Gamma(j+\frac{1}{2})}{\sqrt{2\pi}s^{j+\frac{1}{2}}}+O(\frac{1}{
s^{j}}).
\end{equation}
 \lq\lq Rankin's method\rq\rq (see Proposition 1 of \cite{amo}) is the very useful observation
 that,
 for any $s>0$,
 \begin{equation}
\label{rankin}  \mu(n)\leq e^{ns+g(s)}.
 \end{equation}
  In particular, for $s=
\frac{1}{2n^{2/3}}$ and $j=0$,  we  can combine (\ref{rankin}) and
(\ref{gs})  to  get \begin{equation}\label{lower}\mu(n)\leq
\exp(\frac{3}{2}n^{1/3}+O(1).\end{equation}

Now Odlyzko's Tauberian theorem  can be used to derive a lower
bound for $\mu(n)$. The remainder of this paragraph is a
straightforward verification of the conditions of Theorem 1 in
\cite{amo}.
 Let $s_{*}=s_{*}(n)$ be chosen so as to minimize
$e^{ns+g(s)}$. By calculus and (\ref{gs}),
\begin{equation}
\label{bootstrap} \frac{1}{2n^{2/3}}=s_{*}(1+O(\sqrt{s_{*}})).
\end{equation}
 It is clear  that $s_{*}=o(1).$
We can \lq\lq boostrap\rq\rq twice to get a better estimate.
First, by  replacing
  $O(\sqrt{s_{*}})$  with $o(1)$ in (\ref{bootstrap}), we get
 $s_{*}=O(\frac{1}{n^{2/3}})$. Putting
this rough bound back into (\ref{bootstrap}) again  yields
\begin{equation}
\label{sstar} s_{*}=\frac{1}{2n^{2/3}}(1+O(\frac{1}{n^{1/3}}) ).
\end{equation}
Now let $A_{n}=g^{''}(s_{*}).$ By (\ref{gs}) and (\ref{sstar}),
\begin{equation}
\label{aest}A_{n}= 3 n^{5/3}(1+O(\frac{1}{n^{1/3}})).
\end{equation} Then, by (\ref{gs}), (\ref{sstar}), and
(\ref{aest}), we have
\begin{equation}
\left| g^{'''}(s_{*})\right|\sim 15n^{7/3}=o(A_{n}^{3/2}).
\end{equation}
\par
Having verified the conditions of Odlyzko's theorem, we conclude
that, for all sufficiently large $n$,
\begin{equation}
\label{upper} \mu(n)\geq F(s_{*})\exp\left(
ns_{*}-30s_{*}A_{n}^{1/2}-100\right)=e^{\frac{3}{2}\sqrt[3]{n}+O(n^{1/6})}.
\end{equation}
By Stirling's formula,
\begin{equation}
\label{slow} \frac{1}{\sqrt{8\pi m}}< \frac{ n^{m}}{m!e^{m}} <
\frac{1}{\sqrt{2\pi m}}.
\end{equation}
Finally, putting (\ref{slow}) and (\ref{upper}) and ({\ref{lower})
back into (\ref{AB}), we get

\begin{theorem}
$ E_{n}({\bf B})=e^{\frac{3}{2}\sqrt[3]{n}+O(n^{1/6})}.$
\end{theorem}

\section{Order}

The main goal in this section is the proof of Theorem \ref{muT},
the  estimate for the average period $E_{n}({\bf T})$.  However
first, for comparison and perspective, we prove Proposition
\ref{typical},
 concerning the typical period, that was stated in the
 introduction.
\begin{proof}
Let   $Z(f)=$  denote the number of cyclic vertices $f$ has. By
the  Law of Total Probability,
\begin{equation}
\label{totalpr} P_{n}(\log {\bf B}-\log {\bf T}>\ell_{n}
)=\sum\limits_{m}P_{n}(Z=m)P_{n}(\log {\bf B}-\log {\bf
T}>\ell_{n}|Z=m).
\end{equation}
In the proof of Theorem 8, page 333 of \cite{AT},
 Arratia and Tavare computed the expected value of $\log {\bf
 B}-\log {\bf T}$ given the number of cyclic vertices:
 $E_{n}(\log {\bf B}-\log {\bf T}|Z=m)= O(\log
m(\log\log m)^{2}).$ (Note: their use of the notation $P_{n}$ is
not consistent with the notation in this paper.) Therefore, by
Markov's inequality,there is a constant $c>0$ such that, for all
$\ell >0,$
\begin{equation}
\label{conditional} P_{n}(\log {\bf B}-\log {\bf T}>\ell|Z=m)\leq
\frac{c\log m(\log\log m)^{2}}{\ell}\leq \frac{c\log n(\log\log
n)^{2}}{\ell}.
\end{equation}
Putting (\ref{conditional}) back to the sum (\ref{totalpr}), we
get the proposition.
\end{proof}
Let $Sym([n])\subseteq \Omega_{n}$ be the $n!$ element set of
bijections, and let $M_{n}=E_{n}({\bf T}| Sym([n])) =
\frac{1}{n!}\sum\limits_{\sigma\in Sym([n])}{\bf T}(\sigma)$ be
the average of the orders of the permutations of $[n]$. Given the
set ${\cal Z}$ of cyclic vertices, the restriction of a random $f$
to Sym(${\cal Z}$) is a uniform random permutation of ${\cal Z}.$
Hence
\begin{equation}
\label{summm}
 E_{n}({\bf T})=\sum\limits_{m=1}^{n}P_{n}(Z=m)M_{m}.
\end{equation}

Two helpful theorems area make it possible to estimate this sum.
The first is a simple formula for $P_{n}({Z}=m)$ that appears in
\cite{harris1} and is attributed to Rubin and Sitgreaves.
\begin{theorem}
\label{zpmf} $ P_{n}(Z=m )=\frac{n!m}{ (n-m)! n^{m+1} }$
\end{theorem}
 The second helpful theorem is an  estimate for $M_{m}.$
  Using Erd\H os and Tur\' an's
Tauberian theorem\cite{ET3}, Richard Stong proved in
  \cite{Stong} that
 \begin{theorem}
\label{mun}
  $\log M_{m} =\beta_{0}\sqrt{m/\log m}+O(\frac{\sqrt{m}\log\log
 m}{\log m}). $
 \end{theorem}
 With Theorems \ref{zpmf} and \ref{mun} available, we can prove
 prove
 \begin{theorem}
$\log E_{n}({\bf T})=k_{0}\sqrt[3]{ \frac{n} { \log^{2}
n}}\bigl(1+o(1)\bigr).$
 \end{theorem}
 \begin{proof}
 Define $a_{0}=\sqrt[3]{3I},$ and let
  $m_{0}^{*}=$ the closest integer to $a_{0}\sqrt[3]{\frac{n^{2}}{\log
  n}}$.
For the lower bound, simply let $m=m_{0}^{*}$ in the trivial lower
bound
 $E_{n}({\bf T})\geq  P_{n}
 (Z=m)M_{m}. $Then, by Theorem \ref{zpmf}, Theorem \ref{mun}, and Stirling's
formula, $E_{n}({\bf T})$ is greater than
\[
\exp\left(-\frac{(m_{0}^{*})^{2}}{2n}+
O(\frac{(m_{0}^{*})^{3}}{n^{2}}) +
\beta_{0}\sqrt{\frac{m_{0}^{*}}{\log m_{0}^{*}}}+ O(\frac{
\sqrt{m_{0}^{*}}\log\log m_{0}^{*}}{\log m_{0}^{*}}) \right)
\]

\[
=\exp\left(
\frac{k_{0}n^{1/3}}{\log^{2/3}n}+O(\frac{n^{1/3}\log\log
n}{\log^{7/6}n} ) \right).
\]
For the upper bound, suppose $\epsilon >0$ is a fixed but
arbitrarily small positive number. Define
$\beta_{\epsilon}=\beta+\epsilon,$ and $w_{\epsilon}(m)=\frac{n!}{
(n-m)! n^{m-1} }e^{\beta_{\epsilon}\sqrt{m/\log m}}.$ By Theorem
\ref{mun}, $M_{m}\leq e^{\beta_{\epsilon}\sqrt{m/\log m}}$ for all
sufficiently large $m$.
 Therefore,  for all sufficiently large
$n$,
\begin{equation}
E_{n}({\bf T}) \leq n\max_{m\leq n}P_{n}( Z=m)M_{m}\leq
\max_{m\leq n}w_{\epsilon}(m).
\end{equation}
For $6\leq m\leq n$, let $G_{n,\epsilon}(m)=\log w_{\epsilon}(m)$.
If we write $(n-m)!=\Gamma(n+1-m)$,
 then  $G_{n,\epsilon}(x)$ is twice differentiable for all real numbers
 $x$
 with $6\leq x\leq n.$
 Let  $\Psi(y) = \frac{\Gamma^{'}(y)}{\Gamma(y)}$ be
 the logarithmic derivative of the Gamma
 function so that the first two  derivatives of $G_{n,\epsilon}$ are
\begin{equation}
\label{gprime} G_{n,\epsilon}^{'}(x)= \Psi(n+1-x)-\log n
+\frac{\beta_{\epsilon}}{2 \sqrt{x\log x}}(1-\frac{1}{\log x}),
\end{equation}
and $ $
\begin{equation}
\label{concave}
G_{n,\epsilon}^{''}(x)=-\Psi^{'}(n+1-x)+\frac{\beta_{\epsilon}}{4}\frac{(3-\log^{2}x)}{x^{3/2}\log^{5/2}x}.
\end{equation}
  It is well known \cite{Abram} that
 $\Psi^{'}(y)=\sum\limits_{k=0}^{\infty}\frac{1}{(y+k)^{2}}>0.$
 Thus
 both terms in (\ref{concave}) are negative, and we we  have
$G^{''}_{n,\epsilon}(x)<0$ for $6\leq x\leq n.$   Let
$x_{n,\epsilon}^{*}$ be the unique solution to
$G_{n,\epsilon}^{'}(x)=0$ at which $G_{n,\epsilon}$ attains its
maximum. We need to estimate $x_{n,\epsilon}^{*}$, and then use
that estimate to approximate $G_{n,\epsilon}(x_{n,
\epsilon}^{*}).$
\par
Define
$
m_{n,\epsilon}^{*}=\beta_{\epsilon}^{2/3}\sqrt[3]{3/8}\frac{n^{2/3}}{(\log
n )^{1/3}},$
 and let $\delta\in (0,\frac{1}{2}).$ In order to prove that
 \begin{equation}
 \label{squeeze}
(1-\delta)m_{n,\epsilon}^{*}< x_{n, \epsilon}^{*}<
(1+\delta)m_{n,\epsilon}^{*},
\end{equation}
it suffices to verify that
 $G_{n,\epsilon}^{'}((1-\delta)m_{*} )>0$  and $G_{n,\epsilon
}^{'}((1+\delta)m_{*} )<0.$
 It is well known \cite{Abram}
that
\begin{equation}
\label{psi} \Psi(y)=\log
y+O(\frac{1}{y}).
\end{equation}
Putting (\ref{psi}) in (\ref{gprime}), we get
\begin{equation}
\label{geprime}
 G_{n,\epsilon}^{'}((1-\delta)m_{*} )=
 \frac{      \sqrt[3]{3\beta_{\epsilon}^{2}}
 }
 {\sqrt{1-\delta}\sqrt[3]{8n\log n}}
 \biggl\lbrace
 1    -(1-\delta)^{3/2}+ O(\frac{\log\log n}{\log n}) )
 \biggr\rbrace .
 \end{equation}
 To determine the sign, note that
  $1   -(1-\delta)^{3/2}>\frac{3}{4}\delta.$
 Hence $G_{n,\epsilon}^{'}((1-\delta)m_{*}) >0,$ and we
can even allow  $\delta=\delta_{n}$ to depend on $n$ so long as
$\frac{\delta\log n}{\log\log n}\rightarrow \infty.$ By similar
reasoning $G_{n,\epsilon}^{'}((1+\delta)m_{*}) <0.$ Therefore
 $ x_{n,\epsilon}^{*}=m_{n,\epsilon}^{*}(1+O(\frac{(\log\log n)^{2}}{\log n}).$
 But then, by  Stirling's
formula,
 $  G_{n,\epsilon}(x_{n,\epsilon}^{*})=
\frac{k_{\epsilon} n^{1/3}}{\log^{2/3}n}(1+o(1)),$ where
\[ k_{\epsilon}=  -\frac{(\beta_{\epsilon}^{2/3}\sqrt[3]{3/8})^{2}}{2}+
\beta_{\epsilon}\sqrt{
\frac{\beta_{\epsilon}^{2/3}\sqrt[3]{3/8}}{2/3} }.
\] The theorem now follows from the fact that $\epsilon$ was an
arbitrarily small positive number, and
$\lim\limits_{\epsilon\rightarrow 0^{+}}k_{\epsilon}=k_{0}$.
\end{proof}

\section{Discussion}
 The calculations in this paper depend heavily on the fact that
the probability measure $P_{n}$ is uniform. A separate paper
considers  Ewens-type distributions where the probability of a
permutation or mapping is weighted according  to the number of
components it has \cite{BT},\cite{jh2},\cite{EJS}.
  I do not know if  the results in these papers can be extended to more general
independent choice models  such as those considered by Jaworski in
\cite{J1} or the p-mappings of  Aldous and Pitman,
\cite{AP1},\cite{AMP}.

\end{document}